\documentclass[a4paper,10pt,final]{scrartcl}
\usepackage[english] {babel}
\usepackage[utf8]{inputenc}
\usepackage{a4wide}
\usepackage{amsfonts} 
\usepackage{amssymb}
\usepackage{amsthm}
\usepackage{amsxtra}
\usepackage{mathrsfs}
\usepackage{exscale}
\usepackage{cite}
\DeclareMathOperator*{\argmin}{argmin}
\DeclareMathOperator{\err}{\mathbf{err}}
\newcommand{\Rset}{\mathbb{R}}
\newcommand{\Zset}{\mathbb{Z}}
\newcommand{\EW}[1]{\mathbf{E}\left[#1\right]}

\newcommand{\Prob}[1]{\mathbf{P}\left[#1\right]}
\newcommand{\Var}[1]{\mathbf{Var}\left[#1\right]}
\newcommand{\Xspace}{\mathcal{X}}
\newcommand{\Yspace}{\mathcal{Y}}
\renewcommand{\L}{\mathbf{L}}
\newcommand{\gdag}{g^\dagger}
\newcommand{\gobs}{g^{\rm obs}}
\newcommand{\udag}{u^\dagger}
\newcommand{\ualdel}{u_{\alpha}}
\newcommand{\B}{\mathfrak B}
\newcommand{\calR}{\mathcal{R}}
\newcommand{\g}{\mathfrak g}
\newcommand{\breg}[1]{\mathcal D_{\calR}^{u^*}\left(#1,\udag\right)}
\newcommand{\Cbreg}{C_{\rm bd}}
\newcommand{\Cconc}{C_{\rm conc}}
\newcommand{\T}[2]{\mathcal{T}\left(#2;#1\right)}
\newcommand{\KL}[2]{\mathbb{KL}\left(#2;#1\right)}
\newcommand{\calS}{\mathcal S}
\renewcommand{\S}[2]{\calS\left(#2;#1\right)}
\newcommand{\offset}{\sigma}
\newcommand{\manifold}{\mathbb{M}}
\newcommand{\Cerr}{C_{\rm err}}

\newcommand{\paren}[1]{\left( #1 \right)}
\newcommand{\bracket}[1]{\left[ #1 \right]}
\newcommand{\Zspace}{\mathcal{Y}^{\rm obs}}
\newcommand{\bfj}{\mathbf{j}}

\renewcommand{\limits}{}

\title{Convergence rates in expectation for Tikhonov-type regularization of Inverse Problems with Poisson data}
\author{Frank Werner and Thorsten Hohage\\[0.2cm]\small f.werner@math.uni-goettingen.de, +49 (0)551 39 12468\\\small hohage@math.uni-goettingen.de, +49 (0)551 39 4509\\[0.5cm]
\small Institute for Numerical and Applied Mathematics, University of G\"ottingen\\
\small Lotzestra\ss e 16-18\\
\small 37083 G\"ottingen
}
\date{\today}

\begin{document}
\newtheorem{thm}{Theorem}[section]
\newtheorem{lem}[thm]{Lemma}
\newtheorem{cor}[thm]{Corollary}
\newtheorem{rem}[thm]{Remark}
\newtheorem{ex}[thm]{Example}
\newtheorem{applem}{Lemma}
\renewcommand{\qedsymbol}{\rule{.2cm}{.2cm}}
\newtheorem{ass}{Assumption}

\maketitle

\begin{abstract}
In this paper we study a Tikhonov-type method for ill-posed 
nonlinear operator equations $\gdag = F\left(\udag\right)$ where 
$\gdag$ is an integrable, non-negative function. We assume that
data are drawn from a Poisson process with density $t\gdag$ where
$t>0$ may be interpreted as an exposure time. Such problems occur
in many photonic imaging applications including positron emission tomography,
confocal fluorescence microscopy, astronomic observations, and 
phase retrieval problems in optics. Our approach uses 
a Kullback-Leibler-type data fidelity functional and allows for general 
convex penalty terms. We prove convergence rates of the expectation of
the reconstruction error under a variational source 
condition as $t\to\infty$ both for an a priori and for a 
Lepski{\u\i}-type parameter choice rule. 
\end{abstract}

\section{Introduction}
We consider inverse problems where the ideal data can be interpreted
as a photon density $\gdag\in\L^1(\manifold)$ on some manifold $\manifold$. 
The unknown  will
be described by an element $\udag$ of a subset $\mathfrak{B}$ of a
Banach space $\Xspace$, and $\udag$ and $\gdag$ are related 
by a forward operator $F$ mapping from $\mathfrak{B}$ to $\L^1(\manifold)$:
\begin{equation}\label{eq:opeq} 
F \left(\udag\right) = \gdag\,.
\end{equation}
The data will be drawn from a Poisson process with density $t\gdag$ where
$t>0$ can often be interpreted as an exposure time. Such data can be seen
as a random collection of points on the manifold $\manifold$ on which
measurements are taken (see section \ref{sec:pp} for a precise
definition of Poisson processes). Hence unlike the common deterministic setup
the data do not belong to the same space as the ideal data $\gdag$. 

Such inverse problems occur naturally in photonic imaging since photon
count data are Poisson distributed for fundamental physical reasons. 
Examples include inverse problems in astronomy \cite{BBDV:09},
fluorescence microscopy, in particular 4Pi microscopy \cite{SBH:12},  
coherent X-ray imaging \cite{hw11}, 
and positron emission tomography \cite{CK:02}. 

\smallskip
In this paper we study a penalized likelihood or Tikhonov-type estimator
\begin{equation}\label{eq:ttr}
\ualdel \in\argmin\limits_{u \in \B} \left[\S{F\left(u\right)}{G_t} + \alpha\calR\left(u\right)\right]\,.
\end{equation}
Here $G_t$ describes the observed data, $\mathcal{S}$ is a Kullback-Leibler
type data misfit functional derived in section \ref{sec:pp}, $\alpha>0$
is a regularization parameter, and 
$\mathcal{R}:\Xspace \to \left(-\infty, \infty\right]$ is a convex penalty term,
which may incorporate \textit{a priori} knowledge about the unknown 
solution $\udag$. If $\S{g_2}{g_1} = \|g_1-g_2\|_{\Yspace}^2$ and 
$\calR(u) = \|u-u_0\|_{\Xspace}^2$ with Hilbert space norms $\|\cdot\|_{\Xspace}$ and $\|\cdot\|_{\Yspace}$, 
then \eqref{eq:ttr} is standard Tikhonov regularization. 
In many cases  estimators of the form \eqref{eq:ttr} can
be interpreted as maximum a posteriori (MAP) estimators in a Bayesian 
framework, but our convergence analysis will follow a frequent
paradigm, and in particular $\udag$ will be considered as a deterministic 
quantity. Our data misfit functional $\mathcal S$ will be convex in its second argument, so the minimization problem \eqref{eq:ttr} will be convex
if $F$ is linear.

Recently, considerable progress has been achieved in the deterministic analysis
of variational regularization methods in Banach spaces 
\cite{bh10,bo04,f10,ghs10,g10b,hkps07,s08}. In particular, a number
of papers have been devoted to the Kullback-Leibler divergence as data 
fidelity term in \eqref{eq:ttr}, motivated by the case of Poisson data 
(see \cite{b10,bb11,p08,f10,f11diss,RA:07}), but all of them under deterministic
error assumptions. On the statistical side, inverse problem for Poisson
data have been studied by Antoniadis \& Bigot \cite{ab06} by wavelet Galerkin
methods. Their study is restricted to linear operators with favorable
mapping properties in certain function spaces. Therefore, there
is a need for a statistical convergence analysis for inverse problems
with Poisson data involving more general and in particular nonlinear 
forward operators. This is the aim of the present paper. 

\smallskip 

Our convergence analysis of the estimator \eqref{eq:ttr} is based on  
two basic ingredients:
The first is a a uniform concentration inequality for Poisson data 
(Theorem \ref{thm:pie}), which will be formulated together with some 
basic properties of Poisson processes in Section \ref{sec:pp}. The
proof of Theorem \ref{thm:pie}, which is based on results by 
Reynaud-Bouret \cite{rb03} is given in an appendix. The second
ingredient, presented in Section \ref{sec:det_result}, is a deterministic 
error analysis of \eqref{eq:ttr} for general $\mathcal S$ under a
variational source condition (Theorem \ref{thm:det_conv}). 
Our main results are two estimates of the expected reconstruction error 
as the exposure time $t$ tends to $\infty$: For an a-priori choice
of $\alpha$, which requires knowledge of the smoothness of $\udag$,
such a result in shown in Theorem \ref{thm:cr_random}. 
Finally, a convergence rate result for a completely adaptive method, 
where $\alpha$ is chosen by a Lepski{\u\i}-type balancing principle, is 
presented in Theorem \ref{thm:a_posteriori}. 

\section{Results on Poisson processes}\label{sec:pp}
Let $\manifold\subset \Rset^d$ be a submanifold where measurements are taken,
and let $\{x_1,\dots,x_N\}\subset \manifold$ denote the positions of
the detected photons. Both the total number $N$ of observed photons
and the positions $x_i\in\manifold$ of the photons are random, and 
it is physically evident that the following two properties hold true:
\begin{enumerate}
\item For all measurable subsets $\manifold'\subset\manifold$ the integer
valued random variable $G\left(\manifold'\right):=\#\left\{i~\big|~x_i\in\manifold'\right\}$ 
has expectation $\EW{G\left(\manifold'\right)}=\int_{\manifold'}\gdag\,\mathrm dx$ 
where $\gdag \in L^1(\manifold)$ denotes the underlying photon density. 
\item For any choice of $m$ disjoint measurable subsets $\manifold_1',\dots,\manifold_m'\subset \manifold$ the random variables $G \left(\manifold'_1\right),\dots, G\left(\manifold'_m\right)$ are stochastically independent. 
\end{enumerate}
By definition, this means that $G:=\sum_{i=1}^N\delta_{x_i}$ is a 
Poisson process with intensity $\gdag$. It follows from these properties that 
$G(\manifold')$ for any measurable $\manifold'\subset\manifold$ 
is Poisson distributed with mean $\lambda:= \EW{G(\manifold)}$,
i.e.\ $\Prob{G(\manifold')=n} = \exp(-\lambda)\frac{\lambda^n}{n!}$ 
for all $n\in \{0,1,\dots\}$ 
(see e.g. \cite[Thm 1.11.8]{kmm78}). Moreover, for any measurable $\psi:\Omega\to \Rset$
we have
\begin{align}\label{eq:integralPoisson}
\EW{\int\limits_{\manifold} \psi \,\mathrm dG} = \int\limits_{\manifold} \psi\,\gdag\,\mathrm dx\,,\qquad
\Var{\int\limits_{\manifold} \psi \,\mathrm dG} = \int\limits_{\manifold} \psi^2\,\gdag\,\mathrm dx
\end{align}
whenever the right hand sides are well defined (see \cite{k93}).

Let us introduce for each exposure time $t>0$ a Poisson process $\tilde{G}_t$
with intensity $tg^{\dagger}$ and define $G_t:= \frac{1}{t} \tilde{G}_t$. 
We will study error estimates for approximate solutions to the inverse 
problem \eqref{eq:opeq} with data $G_t$ 
in the limit $t\to \infty$. For this end it will be necessary to derive
estimates on the distribution of the log-likelihood functional
\begin{equation}\label{eq:nllfunct_poisson}
\S{g}{G_t} := 
\int\limits_{\manifold} g \,\mathrm dx 
- \int\limits_{\manifold} \ln g\,\mathrm dG_t
= \int\limits_{\manifold} g \,\mathrm dx - \frac{1}{t} \sum_{i=1}^N \ln g(x_i)\,,
\end{equation}
which is defined for functions $g$ fulfilling $g \geq 0$ a.e. We set $\ln 0 :=-\infty$, so $\S{g}{G_t}=\infty$ if  $g(x_i)=0$ for some 
$i=1,\dots,N$. 
Using \eqref{eq:integralPoisson} we obtain
\begin{align}
\EW{\S{g}{G_t}} = \int\limits_{\manifold} \left[g - \gdag\ln\left(g\right)\right]\,\mathrm dx
\quad \mbox{and}\quad
\Var{\S{g}{G_t}} = \frac{1}{t}\int\limits_{\manifold} \ln\left(g\right)^2\,\gdag\,\mathrm dx
\end{align}
if the integrals exist. Moreover, we have
\[
\EW{\S{g}{G_t}} - \EW{\S{\gdag}{G_t}}
= \int\limits_{\manifold} \bracket{g -\gdag- \gdag \ln \left(\frac{g}{\gdag}\right)}\,\mathrm d x,
\]
and the right hand side (if well-defined) is known as 
\emph{Kullback-Leibler divergence}
\begin{equation}\label{eq:kl}
\KL{g}{\gdag}:=
\int\limits_{\{\gdag>0\}} \left[g-\gdag-\gdag\ln\frac{g}{\gdag}\right]\,\mathrm dx.
\end{equation}
$\KL{g}{\gdag}$ can be seen as the ideal data misfit functional if the
exact data $\gdag$ were known. Since only $G_t$ is given, 
we approximate $\KL{g}{\gdag}$ by $\S{g}{G_t}$ up to the additive 
constant $\EW{\S{\gdag}{G_t}}$, which is independent of $g$. 
The error between the estimated and the
ideal data misfit functional is given by 
\begin{equation}\label{eq:err_Poisson}
\left|\S{g}{G_t}- \EW{\S{\gdag}{G_t}}
-\KL{g}{\gdag}\right| 
= \left|\int\limits_{\manifold} \ln(g) \left(\mathrm dG_t-g^{\dagger} \,\mathrm dx\right)\right|.
\end{equation}

Based on results by  Reynaud-Bouret \cite{rb03}, which can be seen
as an analogue to Talagrand's concentration inequalities for empirical
processes, we will derive the following concentration inequality for
such error terms in the appendix:

\begin{thm}\label{thm:pie}
Let $\manifold \subset \mathbb R^d$ be a bounded domain with Lipschitz boundary $\partial D$, $R \geq 1$ and $s > \frac{d}{2}$. Consider the ball
\[
B_s \left(R\right) := \left\{ \g \in H^s \left(\manifold\right) ~\big|~ \left\Vert \g\right\Vert_{H^s \left(\manifold\right)} \leq R\right\}.
\]
Then there exists a constant $\Cconc \geq 1$ depending only on $\manifold$, 
$s$ and $\|\gdag\|_{\L^1(\manifold)}$ such that
\[
\Prob{\sup\limits_{\g \in B_s \left(R\right)} \left|\int\limits_{\manifold} \g \left( \,\mathrm d G_t - \gdag\,\mathrm d x\right) \right| \leq \frac{\rho}{\sqrt{t}}} \geq 1 - \exp\left(-\frac{\rho}{R \Cconc} \right)
\]
for all $t \geq 1$ and $\rho \geq R \Cconc$. 
\end{thm}

To apply this concentration inequality to the right hand side of
\eqref{eq:err_Poisson}, we would need that $\ln(F(u))\in B_s(R)$
for all $u\in\mathfrak{B}$. However, since 
$\sup_{\mathfrak{g}\in B_s(R)}\|\mathfrak{g}\|_{\infty}<\infty$ 
by Sobolev's embedding theorem,  zeros of $F(u)$
for some $u\in \mathfrak{B}$ would not be admissible, which is a quite
restrictive assumption. Therefore, we use a shifted version of the
data fidelity term with an offset parameter $\sigma>0$:

\begin{align}
\label{eq:defi_Se}
\S{g}{G_t} &:= \int\limits_{\manifold} g \,\mathrm d x -  \int\limits_{\manifold} \ln \left(g+\offset\right) \,(\mathrm d G_t +\offset \mathrm dx) \\[0.1cm]
\T{g}{\gdag} &:= \KL{g+\sigma}{\gdag+\sigma} \label{eq:defi_Te}
\end{align}
Then the error is given by
\begin{equation}\label{eq:def_err}
Z\left(g\right) :=\left|\S{g}{G_t} 
- \EW {\S{\gdag}{G_t}} 
 - \T{g}{\gdag}\right| = \left|\int\limits_{\manifold} \ln \left(g+\offset\right) \left(\mathrm d G_t - \gdag \,\mathrm d x\right)\right|\,.
\end{equation}
We will show in Section \ref{sec:cr_poisson} that Theorem \ref{thm:pie} can
be used to estimate the concentration of $\sup_{u\in\mathfrak{B}} Z(F(u))$
under certain assumptions. 

\section{A deterministic convergence rate result}\label{sec:det_result}

In this section we will perform a convergence analysis for the method \eqref{eq:ttr} with general $\mathcal S$ under a deterministic noise assumption. 
Similar results have been obtained by Flemming \cite{f10,f11diss}, 
Grasmair \cite{g10b}, and Bot \& Hofmann \cite{bh10} under different
assumptions on $\mathcal{S}$. 

As in Section \ref{sec:pp} we will consider the ''distance'' 
between the estimated and the ideal data misfit functional as noise level: 
\begin{ass}\label{ass:SR}
Let $u^\dagger \in \mathfrak B \subset \Xspace$ be the exact solution and denote by $g^\dagger := F\left(u^\dagger\right) \in \Yspace$ the exact data. Let $\Zspace$ be a set containing all possible observations and $\gobs\in \Zspace$ 
the observed data. Assume that:
\begin{enumerate}
\item The exact data fidelity functional 
$\mathcal{T}: F\left(\mathfrak{B}\right) \times \Yspace \to [0,\infty]$ 
is non-negative, and $\mathcal{T}(\gdag,\gdag)=0$. 
\item For the approximate data fidelity term 
$\mathcal{S}:  F\left(\mathfrak{B}\right) \times \Yspace \to [0,\infty]$ 
there exist constants 
$\err\geq 0$ and $\Cerr\geq 1$ such that
\begin{equation}\label{eq:err}
\S{g}{\gobs} -\S{\gdag}{\gobs}  \geq \frac{1}{\Cerr} \T{g}{g^\dagger}- \err
\end{equation}
for all $g \in F\left(\B\right)$.
\end{enumerate}
\end{ass}
\begin{ex}
\begin{itemize}
\item \emph{Classical deterministic noise model:} If $\S{\hat g}{g} = \T{\hat g}{g}  = \left\Vert g-\hat g\right\Vert_{\Yspace}^r$, then we obtain from $\left|a-b\right|^r \geq 2^{1-r} a^r - b^r$ that \eqref{eq:err} holds true with $\Cerr = 2^{r-1}$ and $\err = 2\left\Vert \gdag - \gobs\right\Vert_{\Yspace}^r$. Thus Assumption \ref{ass:SR} covers the classical deterministic noise model.
\item \emph{Poisson data:} For the case of $\mathcal S$ and $\mathcal T$ as in \eqref{eq:defi_Se} and \eqref{eq:defi_Te} 
it can be seen from elementary calculations that \eqref{eq:err} requires $\Cerr = 1$ and
\begin{equation}\label{eq:cond_err_poisson}
\err \geq -\int\limits_{\manifold} \ln\left(\gdag + \offset\right) 
\left(\,\mathrm d G_t - \gdag \,\mathrm d x \right) 
+ \int\limits_{\manifold} \ln\left(F\left(u\right) + \offset\right) \left(\,\mathrm d G_t - \gdag \,\mathrm d x\right)
\end{equation}
for all $u \in \B$. Consequently \eqref{eq:err} holds true with $\Cerr = 1$ if $\err/2$ is an upper bound for the integrals in \eqref{eq:def_err} with $g = F\left(u\right), u \in \B$. We will show that Theorem \ref{thm:pie} ensures
that this holds true for $\err/2=\frac{\rho}{\sqrt{t}}$ with probability $\geq 1-\exp\left(-c\rho\right)$ for some constant $c > 0$ (cf. Corollary \ref{cor:apply_concie_ttr}).
\end{itemize}
\end{ex}
In a previous study of Newton-type methods for inverse problems with 
Poisson data \cite{hw11} the authors had to use a slightly stronger
assumption on the noise level involving a second inequality. 
 \cite[Assumption 2]{hw11} implies \eqref{eq:err} with $\err = \left(1+\Cerr\right) \sup_{u \in \B} \err\left(g\right)$ provided this value 
 is finite. On the other hand, \eqref{eq:err} allows that 
 $\S{g}{\gobs} = \infty$ even if $\T{g}{g^\dagger} < \infty$, which is impossible in \cite[Assumption 2]{hw11} if $\err\left(g\right) < \infty$. 

\smallskip

To measure the smoothness of the unknown solution, we will use a source
condition in the form of a variational inequality, which was introduced by 
Hofmann et al \cite{hkps07} for the case of a H\"older-type source condition with index $\nu = \frac12$ and generalized in many recent publications \cite{bh10,hy10,g10b,f10,fhm11}. For their formulation we need the 
\emph{Bregman distance}. 
For a subgradient $u^* \in \partial \calR \left(\udag\right)\subset\Xspace^*$ (e.g. $u^*=\udag-u_0$ if $\calR\left(u\right) = 1/2\left\Vert u-u_0\right\Vert_{\Xspace}^2$ with a Hilbert norm $\left\Vert\cdot\right\Vert_{\Xspace}$) the Bregman distance of $\calR$ between $u$ and $\udag$ w.r.t.\ $u^*$ is given by 
\[
\breg{u} := \calR \left(u\right) - \calR \left(\udag\right) - \left<u^*, u-\udag\right>.
\]
In the aforementioned example of $\calR\left(u\right) = 1/2\left\Vert u-u_0\right\Vert_{\Xspace}^2$ for a Hilbert space norm $\left\Vert \cdot\right\Vert_{\Xspace}$ we have $\breg{u}
= 1/2\left\Vert u-\udag\right\Vert_{\Xspace}^2$. In this sense, the Bregman distance is a natural generalization of the norm. We will use the Bregman distance also to measure the error of our approximate solutions. 

Now we are able to formulate our assumption on the smoothness of $\udag$:
\begin{ass}[variational source condition]\label{ass:additive_sc}
$\calR :\Xspace \to \left(-\infty,\infty\right]$ is a proper convex functional and there exist $u^*\in\partial \calR\left(\udag\right)$, a parameter $\beta>0$ and an index function $\varphi$ (i.e. $\varphi$ monotonically increasing, $\varphi \left(0\right) = 0$) such that $\varphi^2$ is concave and 
\begin{equation}\label{eq:vieadd}
\beta \breg{u} \leq \calR\left(u\right)-\calR\left(\udag\right) + \varphi\left(\T{F\left(u\right)}{\gdag}\right)\qquad \mbox{for all } u\in\B.
\end{equation}
\end{ass}

\begin{ex}
Let $\psi$ be an index function, $\psi^2$ concave and $F:\B \subset \Xspace \to \Yspace$ Fr\'echet differentiable between Hilbert spaces $\Xspace$ and $\Yspace$ with Fr\'echet derivative $F' \left[\cdot\right]$. 
Flemming \cite{fhm11,f11diss} has shown that
\begin{equation}\label{eq:spectral_sc}
\udag - u_0 = \psi \left(F' \left[\udag\right]^*F' \left[\udag\right]\right) \omega
\end{equation}
together with the tangential cone condition $\left\Vert F' \left[\udag\right] \left(u-\udag\right)\right\Vert_{\Yspace} \leq \eta \left\Vert F\left(u\right) - F\left(v\right)\right\Vert_{\Yspace}$ implies the variational inequality
\begin{equation}\label{eq:vie_spectral}
\beta \left\Vert u-\udag\right\Vert_{\Xspace}^2 \leq \left\Vert u \right\Vert_{\Xspace}^2 - \left\Vert \udag\right\Vert_{\Xspace}^2 + \varphi_\psi \left(\left\Vert F \left(u\right) - \gdag\right\Vert_{\Yspace}^2 \right).
\end{equation}
for all $u \in \B$. Here $\varphi_\psi$ is another index function depending on $\psi$, and for the most important cases of H\"older-type and logarithmic source conditions the implications 
\begin{subequations}\label{eqs:special_ind_funct}
\begin{align}
\psi \left(\tau\right) = \tau^\nu & \qquad\Rightarrow\qquad\varphi_\psi\left(\tau\right) = \tilde \beta \tau^{\frac{2\nu}{2\nu+1}}, \\[0.1cm]
\psi\left(\tau\right) = -\left(\ln \left(\tau\right)\right)^{-p} &\qquad\Rightarrow\qquad\varphi_\psi\left(\tau\right) = \bar \beta \left(-\ln \left(\tau\right)\right)^{-2p}
\end{align}
\end{subequations}
hold true  with some constants $\tilde \beta, \bar \beta$ where $p > 0$ and $\nu \in \left(0,\frac12\right]$ (see Hofmann \& Yamamoto \cite[Prop. 6.6]{hy10} and Flemming \cite[Sec. 13.5.2]{f11diss} respectively).
\end{ex}

With the notation \eqref{eq:def_err} of the error, we are able to perform a deterministic convergence analysis including an error decomposition. 
Following Grasmair \cite{g10b} we use the Fenchel conjugate of $\phi$ 
to bound the approximation error. Recall that the Fenchel conjugate 
$\phi^*$ of a function $\phi: \mathbb R \to \left(-\infty,\infty\right]$ 
is given by
\[
\phi^* \left(s\right) = \sup\limits_{\tau \in \mathbb R}\left(s\tau - \phi \left(\tau\right)\right).
\]
$\phi^*$ is always convex as supremum over the affine-linear (and hence convex) 
functions $s \mapsto s\tau - \phi \left(\tau\right)$. 
Setting $\varphi(\tau):=-\infty$ for $\tau<0$ we obtain
\begin{equation}\label{eq:defi_Fenchel}
\left(-\varphi\right)^* \left(s\right) = \sup\limits_{\tau \geq 0}\left(s\tau + \varphi \left(\tau\right)\right).
\end{equation}
This allows us to apply  tools from convex analysis: 
For convex and continuous $\phi$ Young's inequality holds true (see e.g. \cite[eq. (4.1) and Prop. 5.1]{et76}), which states
\begin{equation}\label{eq:young}
\begin{aligned}
s\tau &\leq \phi\left(\tau\right) + \phi^* \left(s\right)\qquad \text{for all} \qquad s,\tau \in \mathbb R,\\[0.1cm]
s\tau &= \phi\left(\tau\right) + \phi^* \left(s\right)\qquad \Leftrightarrow \qquad \tau \in \partial \phi\left(s\right).
\end{aligned}
\end{equation}
Moreover for convex and continuous $\phi$ we have $\phi^{**} = \phi$ (see e.g. \cite[Prop. 4.1]{et76}). 

\smallskip

Now we are in a position to prove our deterministic convergence rates result:
\begin{thm}\label{thm:det_conv}
Suppose Assumptions \ref{ass:SR} and \ref{ass:additive_sc} hold true and the Tikhonov functional has a global minimizer. Then we have the following assertions:
\begin{enumerate}
\item \label{it:det_conv1} For all $\alpha >0$ and $\err \geq 0$ we have 
\begin{equation}\label{eq:error_decomp}
\beta\breg{\ualdel} \leq \frac{\err}{\alpha} + \left(-\varphi\right)^*\left(-\frac{1}{\Cerr\alpha}\right).
\end{equation}
\item \label{it:det_conv2}
Let $\err>0$. Then the infimum of the right hand side of \eqref{eq:error_decomp}
it attained at $\alpha=\overline{\alpha}$ if and only if 
 $\frac{-1}{\Cerr\overline{\alpha}} \in \partial(-\varphi)(\Cerr\err)$, and
 we have
\begin{equation}\label{eq:rate_apriori}
\beta \breg{u_{\overline{\alpha}}} \leq \sqrt{\Cerr}\varphi\left(\err\right).
\end{equation}
\end{enumerate}
\end{thm}

\begin{proof}
(\ref{it:det_conv1}): By the definition of $\ualdel$ we have
\begin{equation}\label{eq:aux_conv}
\S{F(\ualdel)}{\gobs} + \alpha \calR(\ualdel)
\leq \S{\gdag}{\gobs} + \alpha \calR(\udag).
\end{equation}
It follows that 
\begin{align*}
\beta \breg{\ualdel} \stackrel{\mbox{\scriptsize{Ass.\ \ref{ass:additive_sc}}}}{\quad\leq\quad}&\calR\left(\ualdel\right) - \calR\left(\udag\right) + \varphi\left(\T{F\left(\ualdel\right)}{\gdag}\right)\\[0.1cm]
\stackrel{\eqref{eq:aux_conv}}{\quad\leq\quad}& \frac{1}{\alpha} \left(\S{\gdag}{\gobs}-\S{F\left(\ualdel\right)}{\gobs} \right) + \varphi\left(\T{F(\ualdel)}{\gdag}\right)\\[0.1cm]
\stackrel{\mbox{\scriptsize{Ass.\ \ref{ass:SR}}}}{\quad\leq\quad}& \frac{\err}{\alpha} - \frac{1}{\Cerr\alpha}\T{F(\ualdel)}{\gdag} + \varphi\left(\T{F\left(\ualdel\right)}{\gdag}\right) \\[0.1cm]
\quad\leq\quad & \frac{\err}{\alpha}  + \sup\limits_{\tau\geq 0} \left[ \frac{\tau}{-\Cerr \alpha} - \left(-\varphi\right)\left(\tau\right)\right] \\[0.1cm]
\stackrel{\eqref{eq:defi_Fenchel}}{\quad=\quad}& \frac{\err}{\alpha} + (-\varphi)^*\left(-\frac{1}{\Cerr\alpha}\right).
\end{align*}
(\ref{it:det_conv2}): Using the fact that $\left(-\varphi\right)^{**}=-\varphi$ we obtain
\begin{align*}
\inf\limits_{\alpha> 0}\left[ \frac{\err}{\alpha} + \left(-\varphi\right)^*\left(-\frac{1}{\Cerr\alpha}\right)\right] = & -\sup\limits_{s< 0} \left[s\Cerr\err - \left(-\varphi\right)^*\left(s\right)\right] \\[0.1cm]
= & - \left(-\varphi\right)^{**}\left(\Cerr\err\right) 
=\varphi(\Cerr\err) \leq \sqrt{\Cerr}\varphi\left(\err\right)
\end{align*}
where we used the concavity of $\varphi^2$. By the conditions for equality in Young's inequality \eqref{eq:young}, the supremum is attained at 
$\alpha=\overline{\alpha}$ if and only if $\frac{-1}{\Cerr\overline{\alpha}} \in \partial\left(-\varphi\right)\left(\Cerr\err\right)$.
\end{proof}

\begin{rem}
Since $\varphi$ is assumed to be finite,  we have $\partial(- \varphi) \left(s\right)\neq \emptyset$ for all $s>0$ (see e.g. \cite[Cor. 2.3 and Prop. 5.2]{et76}), i.e.\ the parameter choice \eqref{eq:stop_a_priori} is feasible. 
If $\varphi$ is differentiable, then 
$\partial \left(-\varphi\right) \left(s\right) = \left\{ -\varphi' \left(s\right)\right\}$ and \eqref{eq:stop_a_priori} is equivalent to $\alpha = 1/(\Cerr \varphi'\left(\Cerr\err\right))$.
\end{rem}

\begin{ex}[Classical case]
Let $F = T : \Xspace \to \Yspace$ be a bounded linear operator between Hilbert spaces $\Xspace$ and $\Yspace$. For $\S{g_1}{g_2} = \T{g_1}{g_2} = \left\Vert g_1 - g_2\right\Vert_{\Yspace}^2$ and $\calR\left(u\right) = \left\Vert u-u_0\right\Vert_{\Xspace}^2$ we have $\breg{u} =\left\Vert u-\udag\right\Vert_{\Xspace}^2$, $\Cerr =2$ and $\err = 2\left\Vert \gdag-\gobs\right\Vert_{\Yspace}^2$. Moreover \eqref{eq:spectral_sc} implies \eqref{eq:vieadd} with $\varphi = \varphi_\psi$.

If $\left\Vert \gdag-\gobs\right\Vert_{\Yspace} \leq \delta$ as mentioned in the introduction, then we obtain for an appropriate parameter choice $\left\Vert \ualdel-\udag\right\Vert_{\Xspace} = \mathcal O\left( \sqrt{ \varphi_\psi \left(\delta^2\right) }\right)$. For the special examples of $\psi$ given in \eqref{eqs:special_ind_funct} we obtain
\[
\left\Vert \ualdel-\udag\right\Vert_{\Xspace} = \mathcal O \left(\delta^{\frac{2\nu}{2\nu+1}} \right), \qquad \left\Vert \ualdel-\udag\right\Vert_{\Xspace} = \mathcal O \left(\left(-\ln\left(\delta\right)\right)^{-p}\right)
\]
respectively, and these convergence rates are known to be of optimal order.
\end{ex}

\section{Convergence rates for Poisson data with a-priori parameter choice rule}\label{sec:cr_poisson}

In this section we will combine Theorems \ref{thm:pie} 
and \ref{thm:det_conv} to obtain convergence rates for the method 
\eqref{eq:ttr} with Poisson data.  We need the following properties 
of the operator $F$:
\begin{ass}[Assumptions on the forward operator]\label{ass:F}
Let $\Xspace$ be a Banach space and $\B \subset \Xspace$ a bounded, closed and convex subset containing the exact solution $\udag$ to \eqref{eq:opeq}. Let $\manifold \subset \mathbb R^d$ a bounded domain with Lipschitz boundary $\partial D$. Assume moreover that the operator $F : \B \to \Yspace :=\L^1 \left(\manifold\right)$ has the following properties:
\begin{enumerate}
\item \label{it:F_pos} $F \left(u\right) \geq 0$ a.e.\ for all $u \in \B$.
\item \label{it:F_Hs} There exists a Sobolev index $s > \frac{d}{2}$ such that 
$F(\mathfrak{B})$ is a bounded subset of $H^s \left(\manifold\right)$.
\end{enumerate}
\end{ass}
Property (\ref{it:F_pos}) is natural since photon densities 
(or more generally intensities of Poisson processes) have to be non-negative. Property (\ref{it:F_Hs}) is not restrictive for inverse problems  
since it corresponds 
to a smoothing property of $F$ which is usually responsible for the 
ill-posedness of the underlying problem. 

\begin{rem}[Discussion of Assumption \ref{ass:additive_sc}]\label{rem:vie_poisson}
Let Assumption \ref{ass:F} and \eqref{eq:spectral_sc} hold true. Since we have the lower bound
\begin{equation}\label{eq:kl_lower_bound}
\left\Vert g-\hat g\right\Vert_{\L^2 \left(\manifold\right)}^2 \leq \left(\frac43 \left\Vert g + \offset\right\Vert_{\L^\infty\left(\manifold\right)} + \frac23 \left\Vert \hat g + \offset\right\Vert_{\L^\infty\left(\manifold\right)} \right) \T {g}{\hat g}
\end{equation}
with $\mathcal T$ as in \eqref{eq:defi_Te} at hand (see \cite{bl91}), \eqref{eq:vie_spectral} obviously implies \eqref{eq:vieadd} with $\mathcal T$ as in \eqref{eq:defi_Te} and an index function differing from $\varphi_\psi$ only by a multiplicative constant.

Thus Assumption \ref{ass:additive_sc} is weaker than a spectral source 
condition. In particular, if $F\left(\udag\right) = 0$ on some parts 
of $\manifold$ it may happen that \eqref{eq:vieadd} holds true with an index function better than $\varphi_\psi$.
\end{rem}

Assumption \ref{ass:F} moreover allows us to prove the following corollary, which shows that Theorem \ref{thm:pie} applies for the integrals in \eqref{eq:def_err}:
\begin{cor}\label{cor:apply_concie_ttr}
Let Assumption \ref{ass:F} hold true,  set
\[
R := \sup\limits_{u \in \B} \left\Vert F \left(u\right)\right\Vert_{H^s \left(\manifold\right)}\,,
\]
and consider $Z$ defined in \eqref{eq:def_err} with $\offset > 0$.
Then there exists $\Cconc\geq 1$ depending only on $\manifold$ and $s$ such that
\begin{equation}\label{eq:concie_ppp_ttr}
\Prob{\sup\limits_{u \in \B} Z\left(F\left(u\right)\right) \leq \frac{\rho}{\sqrt{t}}} \geq 1 -\exp\left(-\frac{\rho}{R\max \left\{\offset^{-{\lfloor s \rfloor-1}}, \left|\ln \left(R\right)\right|\right\}\Cconc}\right)
\end{equation}
for all $t \geq 1, \rho \geq R\max \left\{\offset^{-{\lfloor s \rfloor-1}},\left|\ln \left(R\right)\right|\right\}\Cconc$.
\end{cor}   
\begin{proof}
W.l.o.g we may assume that $R \geq 1$. Due to Sobolev's embedding theorem and $s > d/2$ we have $\left\Vert F\left(u\right) \right\Vert_{\L^\infty\left(\manifold\right)} \leq R \left\Vert E_{\infty} \right\Vert  $ for all $u \in \B$.

By an extension argument it can be seen from \cite{ms02} that for $\manifold \subset \mathbb R^d$ with Lipschitz boundary, $g \in H^s\left(\manifold\right)\cap \L^\infty \left(\manifold\right)$ and $\Phi \in C^{\lfloor s \rfloor+1} \left(\mathbb R\right)$ one has $\Phi \circ g \in H^s \left(\manifold\right)$ and 
\begin{equation}\label{eq:sobo_norm_est}
\left\Vert \Phi \circ g \right\Vert_{H^s \left(\manifold\right)} \leq C \left\Vert \Phi\right\Vert_{C^{\lfloor s \rfloor+1} \left(\mathbb R\right)} \left\Vert g\right\Vert_{H^s \left(\manifold\right)}
\end{equation}
with $C >0$ independent of $\Phi$ and $g$. 
To apply this result, we first extend the function $x\mapsto\ln\left(x + \offset\right)$ from 
$\left[0, R \left\Vert E_{\infty} \right\Vert  \right]$ (since we have 
$0 \leq F\left(u\right) \leq R\left\Vert E_{\infty} \right\Vert $ a.e.)  
to a function $\Phi$ on the whole real line such that $\Phi \in C^{\lfloor s \rfloor+1} \left(\mathbb R\right)$. Then for any fixed $u \in \B$ we obtain $\Phi \circ F\left(u+\offset\right) \in H^s \left(\manifold\right)$ and since $\Phi_{|_{\left[0, R\left\Vert E_{\infty} \right\Vert \right]}}
\left(\cdot\right) = \ln\left(\cdot + \offset\right)$ and $0 \leq F\left(u\right) \leq R \left\Vert E_{\infty} \right\Vert $ a.e., we have
\(
\Phi \circ \left(F\left(u\right)+\offset\right) 
= \ln \left(F \left(u\right)+\offset\right)
\) a.e. 
Since all derivatives up to order ${\lfloor s \rfloor+1}$ of $x \mapsto \ln \left(x + \offset\right)$ and hence of $\Phi$ on $\left[0, R\left\Vert E_{\infty} \right\Vert \right]$ can be bounded by some constant of order $\max \left\{\offset^{-{\lfloor s \rfloor-1}}, \ln \left(R\left\Vert E_{\infty} \right\Vert \right) \right\}$, the extension and composition procedure described above is bounded, i.e. there exists by \eqref{eq:sobo_norm_est} a constant $\tilde C>0$ independent of $u, R$ and $\offset$ such that
\begin{align*}
\left\Vert \ln \left(F \left(u\right)+\offset\right) \right\Vert_{H^s \left(\manifold\right)}&\leq \tilde C\max \left\{\offset^{-{\lfloor s \rfloor-1}}, \ln \left(R\right) \right\}R
\end{align*}
for all $u \in \B$. Now the assertion follows from Theorem \ref{thm:pie}.
\end{proof}

Now we are able to present our first main result for Poisson data:
\begin{thm}\label{thm:cr_random}
Let the Assumptions \ref{ass:additive_sc} with $\mathcal T$ defined in \eqref{eq:defi_Te} and Assumption \ref{ass:F} be satisfied. 
Moreover, suppose that \eqref{eq:ttr} with $\mathcal{S}$ in 
\eqref{eq:defi_Se} has a global minimizer. 
If we choose the regularization parameter $\alpha = \alpha \left(t\right)$ 
such that 
\begin{equation}\label{eq:stop_a_priori}
\frac{1}{\alpha}  \in -\partial\left(-\varphi\right) \left(\frac{1}{\sqrt{t}}\right)
\end{equation}
then we obtain the convergence rate
\[
\EW{\breg{\ualdel}} = \mathcal O \left(\varphi\left(\frac{1}{\sqrt{t}}\right)\right),\qquad t \to \infty.
\]
\end{thm}
\begin{proof}
First note that Assumption \ref{ass:SR} holds true with $\Cerr = 1$ whenever the bound $\err$ fulfills \eqref{eq:cond_err_poisson}. By Corollary \ref{cor:apply_concie_ttr} the right-hand side of \eqref{eq:cond_err_poisson} is bounded by $2\frac{\rho}{\sqrt{t}}$ with probability greater or equal $1- \exp\left(-c\rho\right)$ with $\rho\geq 1/c$, 
\[
c = \left(R \max\left\{\offset^{-{\lfloor s \rfloor-1}}, \left|\ln\left(R\right)\right| \Cconc\right\}\right)^{-1}\,,
\]
and $\Cconc$ as in Corollary \ref{cor:apply_concie_ttr}. 
Now let $\rho_k:= c^{-1} k, k\in \mathbb N$ and consider the events
\[
E_0 := \emptyset, \qquad E_k := \left\{\sup\limits_{u \in \B} Z\left(F\left(u\right)\right) \leq \frac{\rho_k}{\sqrt{t}}\right\}, \qquad k \in \mathbb N
\]
with $Z$ as defined in \eqref{eq:def_err}. Corollary \ref{cor:apply_concie_ttr} implies
\[
\Prob{E_k^c} \leq \exp\left(-k\right)
\]
and on $E_k$ Assumption \ref{ass:SR} holds true with $\Cerr = 1$ and $\err = 2 \sup_{u \in \B} Z\left(F\left(u\right)\right) \leq 2\rho_k / \sqrt{t}$. Thus Theorem \ref{thm:det_conv}(\ref{it:det_conv1}) implies
\[
\max\limits_{E_k} \breg{\ualdel} \leq\frac{1}{\beta}\left(\left(-\varphi\right)^* \left(-\frac{1}{\alpha} \right) + \frac{2\rho_k}{\alpha\sqrt{t}}\right)
\leq  \frac{2 \rho_k}{\beta}\left(\left(-\varphi\right)^* \left(-\frac{1}{\alpha} \right) + \frac{1}{\alpha\sqrt{t}}\right)
\]
for all $k \in \mathbb N$ and $\alpha > 0$. According to 
Theorem~\ref{thm:det_conv}(\ref{it:det_conv2}) the infimum of the right 
hand side is attained at $\alpha$ defined in \eqref{eq:stop_a_priori}, and 
\[
\max\limits_{E_k} \breg{\ualdel} \leq \frac{2\rho_k}{\beta} \varphi \left(\frac{1}{\sqrt{t}}\right)
\]
for all $k \in \mathbb N$ with $C\left(k\right) =  \frac{2}{\beta}c^{-1} k$. Now we obtain
\begin{align*}
\EW{\breg{u_\alpha}}  &=\sum\limits_{k=1}^{\infty} \Prob{E_k\setminus E_{k-1}}
 \EW{\breg{u_\alpha} ~\bigg|~ E_k \setminus E_{k-1}} \\[0.1cm]
& \leq \sum\limits_{k=1}^{\infty} \Prob{E_k\setminus E_{k-1}} 
\max\limits_{E_k} \breg{\ualdel} \\[0.1cm]
& \leq \Prob{E_1} \frac{2\rho_1}{\beta}\varphi\left(\frac{1}{\sqrt{t}}\right)+ \sum\limits_{k=2}^{\infty} \Prob{E_{k-1}^c} \frac{2\rho_k}{\beta}
\varphi \left(\frac{1}{\sqrt{t}}\right)\\[0.1cm]
& \leq \frac{2}{\beta} c^{-1} \left(\sum\limits_{k=1}^{\infty} \exp\left(-\left(k-1\right)\right) k\right)\varphi\left(\frac{1}{\sqrt{t}}\right).
\end{align*}
The sum converges and the proof is complete.
\end{proof}

\section{A Lepski{\u\i}-type parameter choice rule}\label{sec:lepskij}
Usually the parameter choice rule \eqref{eq:stop_a_priori} is not implementable
since it requires a priori knowledge of the function $\varphi$ characterizing
the smoothness of the unknown solution $\udag$. To adapt to unknown smoothness
of the solution, a posteriori parameter choice rules have to be used.
In a deterministic context the most widely used such rule is 
\emph{discrepancy principle}. However, in our context
is not applicable in an obvious way since $\mathcal S$
approximates $\mathcal T$ only up to the unknown constant 
$\EW{\S{\gdag}{G_t}}$.

In the following we will describe and analyse the Lepsk{\u\i} principle
as described and analyzed in the context of inverse problems by
Math\'e and Pereverzev \cite{mp03,m06}.
Lepski{\u\i}'s balancing principle requires a metric on $\Xspace$, and hence we assume in the following that there exists a constant $\Cbreg > 0$ and a number $q \geq 1$ such that
\begin{equation}\label{eq:R_for_Lepskii}
\left\Vert u-\udag\right\Vert_{\Xspace}^q \leq \Cbreg \breg{u} \qquad \text{for all} \qquad u \in \B.
\end{equation}
This is fulfilled trivially with $q=2$ and $\Cbreg = 1$ if $\Xspace$ is a Hilbert space and $\calR\left(u\right) = \left\Vert u-u_0\right\Vert_{\Xspace}^2$ (then we have equality in \eqref{eq:R_for_Lepskii}). Moreover for a $q$-convex Banach space $\Xspace$ and $\calR\left(u\right) = \left\Vert u\right\Vert_{\Xspace}^q$ the estimate \eqref{eq:R_for_Lepskii} is valid (see \cite{XR:91}). 
Besides this special cases of norm powers, \eqref{eq:R_for_Lepskii} can be fulfilled for other choices of $\calR$. E.g.\ for maximum entropy  regularization, i.e. $\calR \left(u\right) = \int_a^b u \ln \left(u\right) \,\mathrm d x$, the Bregman distance coincides with the Kullback-Leibler divergence, and we have seen in Remark \ref{rem:vie_poisson} that  \eqref{eq:R_for_Lepskii} holds true in this situation.

\smallskip

The deterministic convergence analysis from Section \ref{sec:det_result} already provides an error decomposition. 
Assuming $\beta \geq 1/2$, 
Theorem \ref{thm:det_conv}(\ref{it:det_conv1}) together with \eqref{eq:R_for_Lepskii} states that
\begin{equation}\label{eq:error_dec_final}
\left\Vert u_\alpha - \udag\right\Vert_{\Xspace} \leq \frac12 \left(f_{\rm app} \left(\alpha\right) + f_{\rm noi} \left(\alpha\right)\right)\qquad\text{for all}\qquad \alpha > 0
\end{equation}
with the \emph{approximation error} $f_{\rm app}^\beta\left(\alpha\right)$ and the \emph{propagated data noise error} $f_{\rm noi}^\beta\left(\alpha\right)$ defined by
\begin{equation}\label{eq:error_functions}
f_{\rm app} \left(\alpha\right) := 
2\left(2\Cbreg\left(-\varphi\right)^* \left(- \frac{1}{\alpha}\right)\right)^{\frac1q}
\quad\text{and}\quad 
f_{\rm noi} \left(\alpha\right) := 2\left(2\Cbreg\frac{\err}{\alpha}\right)^{\frac1q}\,.
\end{equation}
Here the constant $2$ in front of $\Cbreg$ is an estimate of $1/\beta$. 
For the error decomposition \eqref{eq:error_dec_final} it is important to note that $f_{\rm app}$ is typically unknown, whereas $f_{\rm noi}$ is known if the upper bound $\err$ is available. 
But due to Corollary \ref{cor:apply_concie_ttr} the error is bounded by $\rho / \sqrt{t}$ with probability $1-\exp\left(-c\rho\right)$. This observation
is fundamental in the proof of the following theorem:
\begin{thm}\label{thm:a_posteriori}
Let Assumptions \ref{ass:additive_sc} and \ref{ass:F} with $\beta \in \left[\frac{1}{2}, \infty\right)$ and $\mathcal S$ and 
$\mathcal T$ as in \eqref{eq:defi_Se} and \eqref{eq:defi_Te} be fulfilled and suppose \eqref{eq:R_for_Lepskii} holds true. Suppose that  \eqref{eq:ttr} has a global minimizer and let $\offset > 0$, $r > 1$, $R := \sup_{u \in \B} \left\Vert F\left(u\right)\right\Vert_{H^s \left(\manifold\right)} < \infty$ and $\tau \geq \frac14 R \max \left\{\offset^{-{\lfloor s \rfloor-1}},\left|\ln\left(R\right)\right|\right\} \Cconc$. Define the sequence
\[
\alpha_j := \frac{\tau \ln\left(t\right)}{\sqrt{t}}r^{2j-2}, \qquad j \in \mathbb N.
\]
Then with $m := \min\left\{j \in \mathbb N ~\big|~ \alpha_j \geq 1\right\}$ the choice 
\begin{equation}\label{eq:stop_a_posteriori}
j_{\rm bal} := \min \left\{j \in  \left\{1, ..., m\right\} ~\big|~ \left\Vert u_{\alpha_i} - u_{\alpha_j}\right\Vert \leq 4 \left(4 \Cbreg\right)^{\frac1q} r^{\frac{2-2i}{q}} \text{ for all } i < j \right\} 
\end{equation}
yields
\[
\EW {\left\Vert u_{\alpha_{j_{\rm bal}}} - \udag\right\Vert_{\Xspace}^q} = \mathcal O \left(\varphi \left(\frac{\ln\left(t\right)}{\sqrt{t}}\right)\right) \qquad\text{as}\qquad t\to \infty.
\]
\end{thm}
\begin{proof}
If $t \geq \exp\left(4\right)$, the assumptions of Corollary \ref{cor:apply_concie_ttr} are fulfilled with $\rho\left(t\right) := \tau \ln \left(t\right)$. Then with $Z$ as in \eqref{eq:def_err} the event
\[
A_\rho := \left\{ \sup\limits_{u \in \B} Z\left(F\left(u\right)\right) \leq \frac{\rho\left(t\right)}{\sqrt{t}}\right\}
\]
has probability  $\Prob{A_\rho^c} \leq \exp\left(-c\rho\left(t\right)\right)$ with $c =\left(R \max \left\{\offset^{-{\lfloor s \rfloor-1}}, \left|\ln\left(R\right)\right|\right\} \Cconc\right)^{-1}$ by Corollary \ref{cor:apply_concie_ttr}. Moreover, as we have seen in \eqref{eq:error_dec_final}, on $A_\rho$ the error decomposition
\[
\left\Vert u_j- \udag\right\Vert_{\Xspace} \leq \frac12 \left(\phi\left(j\right) + \psi\left(j\right)\right)
\]
holds true with
\[
\psi\left(j\right)= 2 \left(4 \Cbreg\right)^{\frac1q} \left(\frac{\rho\left(t\right)}{\sqrt{t}\alpha_j}\right)^{\frac1q}= 2 \left(4 \Cbreg\right)^{\frac1q} r^{\frac{2-2j}{q}} 
\]
and $\phi = f_{\rm app}$ as in \eqref{eq:error_functions}. Note that $2 \psi \left(i\right)$ corresponds to the required bound for $\left\Vert u_{\alpha_i} - u_{\alpha_j}\right\Vert_{\Xspace}$ in \eqref{eq:stop_a_posteriori}. The function $\psi$ is obviously non-increasing and fulfills $\psi \left(j\right) \leq r^{2/q}\left(j+1\right)$ and it can be seen by elementary computations that $\phi$ is monotonically increasing. Now \cite[Cor. 1]{m06} implies the so-called oracle inequality
\[
\max\limits_{A_\rho} \left\Vert u_{\alpha_{j_{\rm bal}}} - \udag\right\Vert_{\Xspace} \leq 3 r^{\frac{2}{q}} \min\left\{ \phi\left(j\right) + \psi\left(j\right) ~\big|~ j\in \left\{1, ..., m\right\}\right\}. 
\]
By inserting the definitions of $\phi$ and $\psi$ we find
\begin{equation}\label{eq:proof_lepskij_aux}
\max\limits_{A_\rho} \left\Vert u_{\alpha_{j_{\rm bal}}} - \udag\right\Vert_{\Xspace}^q \leq 4 r^2 12^q \Cbreg \min\limits_{j=1,...,m} \left(\left(-\varphi\right)^* \left(-\frac{1}{\alpha_j}\right) +  \frac{\rho\left(t\right)}{\sqrt{t}\alpha_j} \right)
\end{equation}
and obviously the minimum over $\alpha_1, ...,\alpha_m$ can be replaced up to some constant depending only on $r$ by the infimum over $\alpha \geq \alpha_1$
if $t$ is sufficiently large. By Theorem \ref{thm:det_conv}(\ref{it:det_conv2})
the sum $\left(-\varphi\right)^* \left(-1/\alpha\right) +  \rho\left(t\right)/\left(\sqrt{t}\alpha\right)$ attains its minimum 
over $\alpha\in (0,\infty)$ 
at $\alpha =\alpha_{\rm opt}$ if and only if 
$1/\alpha_{\rm opt} \in-\partial\left(-\varphi\right) \left(\rho\left(t\right)/\sqrt{t}\right)$. 
Note that $\rho\left(t\right)/\sqrt{t} = \alpha_1$. By elementary arguments from convex analysis we find using the concavity of $\varphi$ that
\[
\frac{1}{\alpha_{\rm opt}} \leq-\inf\partial\left(-\varphi\right) \left(\alpha_1\right) = \lim\limits_{h \searrow 0} \frac{\varphi\left(\alpha_1\right) - \varphi \left(\alpha_1-h\right)}{h} \leq \frac{\varphi\left(\alpha_1\right) - \varphi\left(s\right)}{\alpha_1 - s}
\]
for all $0 \leq s \leq \alpha_1$. Thus choosing $s = 0$ shows that $\alpha_1 / \alpha_{\rm opt} \leq \varphi \left(\alpha_1 \right) = \varphi \left(\rho\left(t\right)/\sqrt{t}\right)$ for all $t >0$. As the right-hand side decays to $0$ as  $t \to \infty$, we have $\alpha_1 \leq \alpha_{\rm opt}$ 
for $t$ sufficiently large. Therefore, 
the minimum in \eqref{eq:proof_lepskij_aux} can indeed be replaced 
(up to some constant) by the infimum over all $\alpha >0$ 
(see \cite[Lem. 3.42]{w12} for details). 
Defining $\text{diam}\left(\mathfrak B\right) := \sup_{u,v \in \mathfrak B} \left\Vert u - v\right\Vert_{\Xspace}$ which is finite 
by Assumption \ref{ass:F} we find from \eqref{eq:proof_lepskij_aux} and 
Theorem \ref{thm:det_conv} that
\begin{align*}
\EW{\left\Vert u_{n_{\rm bal}} - \udag\right\Vert_{\Xspace}^q} &\leq \Prob{A_\rho} \max\limits_{A_\rho} \left\Vert u_{\alpha_{j_{\rm bal}}} - \udag\right\Vert_{\Xspace}^q + \Prob{A_\rho^c} \max\limits_{A_\rho^c} \left\Vert u_{\alpha_{j_{\rm bal}}} - \udag\right\Vert_{\Xspace}^q  \\[0.1cm]
& \leq C\varphi\left(\frac{\ln\left(t\right)}{\sqrt{t}}\right) + \exp\left(-c\rho\left(t\right)\right) \text{diam} \left(\mathfrak B\right)^q.
\end{align*}
with some constant $C >0$. 
Due to the definition of $\rho$, $2\tau c \geq \frac12$, $\ln\left(t\right) \geq 1$ and $\ln\left(t\right)/\sqrt{t} <1$ we obtain
\[
\exp\left(-c\rho\left(t\right)\right) = \left(\frac{1}{\sqrt{t}}\right)^{2\tau c} \leq \left(\frac{\ln\left(t\right)}{\sqrt{t}}\right)^{2\tau c} \leq \sqrt{\frac{\ln\left(t\right)}{\sqrt{t}}} \leq \frac{1}{\varphi\left(1\right)} \varphi\left(\frac{\ln\left(t\right)}{\sqrt{t}}\right)
\]
using  the concavity of $\varphi^2$. This proves the assertion.
\end{proof}

Note that the constants $R$ and $\Cconc$ - which are necessary to ensure a proper choice of the sequence $\alpha_j$ and hence for the implementation of this Lepski{\u\i}-type balancing principle - can be calculated in principle
(assuming e.g.\ the scaling condition $\|\gdag\|_{\L^1(\manifold)}=1$). 
Thus Theorem \ref{thm:a_posteriori} yields convergence rates in expectation for a completely adaptive algorithm.

\smallskip

Comparing the rates in Theorems \ref{thm:cr_random} and \ref{thm:a_posteriori} 
note that we have to pay a logarithmic factor for adaptation to 
unknown smoothness by the Lepski{\u\i} principle. It is known 
(see \cite{t00}) that in some cases the loss of such a logarithmic
factor is inevitable.

\appendix

\section{Proof of Theorem \ref{thm:pie}}\label{sec:app}

\renewcommand{\thesection}{\Alph{section}}

\newcommand{\Cgdag}{c_1}
\newcommand{\Cext}{c_2}
In this section we will prove the uniform concentration inequality stated in Theorem \ref{thm:pie}. 
Our result is based on the work of Reynaud-Bouret \cite{rb03} who proved the following concentration inequality:
\begin{lem}[{\cite[Corollary 2]{rb03}}]\label{lem:concie}
Let $N$ be a Poisson process with finite mean measure $\nu$. Let $\left\{f_a\right\}_{a \in A}$ be a countable family of functions with values in $\left[-b,b\right]$ and define 
\[
Z := \sup\limits_{a \in A} \left|\int\limits_{\manifold} f_a\left(x\right) \left(\mathrm d N - \mathrm d \nu\right) \right| \qquad \text{and}\qquad v_0 := \sup\limits_{a \in A} \int\limits_{\manifold} f_a^2\left(x\right) \,\mathrm d \nu.
\]
Then for all positive numbers $\rho$ and $\varepsilon$ it holds
\[
\Prob {Z \geq \left(1+ \varepsilon\right) \EW{Z} + \sqrt{12 v_0 \rho} + \kappa \left(\varepsilon\right) b\rho} \leq \exp\left(-\rho\right)
\]
where $\kappa \left(\varepsilon\right) = 5/4 + 32 / \varepsilon$.
\end{lem}
We will use a denseness argument to apply Lemma \ref{lem:concie} to 
$t\tilde{Z}$ with 
\[
\tilde Z := \sup\limits_{\g \in B_s \left(R\right)}  \left|\int\limits_{\manifold} \g\left(x\right) \left(\,\mathrm d G_t - \gdag\,\mathrm d x\right) \right|\,.
\]
The properties derived in the following lemma will be sufficient to bound
$\EW{\tilde{Z}}$:
\begin{lem}\label{lem:sum}
Let $\manifold \subset \mathbb R^d$ be a bounded domain with Lipschitz boundary, $R>0$ and suppose $s>\frac{d}{2}$. Then there exists a countable 
family of real-valued functions $\{\phi_\bfj : \bfj \in \mathcal{J}\}$, 
numbers $\gamma_\bfj$, $\bfj\in \mathcal{J}$ and constants $\Cgdag,\Cext>0$ 
depending only on $s$ and $\manifold$ such  that 
\begin{align}\label{eq:cond_gamma}
&\sum\limits_{\bfj\in \mathcal{J}} \gamma_\bfj^2 \int\limits_{\manifold} \phi_\bfj^2 \gdag \,\mathrm d x  \leq \Cgdag \|\gdag\|_{\L^1(\manifold)}\,,
\end{align}
and for all $\g \in B_s\left(R\right)$ there exists 
real numbers $\beta_\bfj, \bfj \in \mathcal{J}$ such that 
\begin{align}
\label{eq:sum_representation}
& \g =  \sum\limits_{\bfj\in \mathcal{J}} \beta_\bfj \phi_\bfj \quad 
\mbox{and}\quad
\sum\limits_{\bfj\in \mathcal{J}} \left(\frac{\beta_\bfj}{\gamma_\bfj}\right)^2 
\leq \Cext^2 R^2\,.
\end{align}
\end{lem}
\begin{proof}
Choose some $\kappa > 0$ such that $\overline{\manifold} \subset \left(-\kappa,\kappa\right)^d$. Then there exists a continuous extension operator 
$E: H^s \left(\manifold\right) \longrightarrow 
H^s_0\left(\left[-\kappa, \kappa\right]^d\right)$ (see e.g.\ 
\cite[Cor.~5.1]{w87}). Consider the following orthonormal bases
$\{\varphi_j:j\in\Zset\}$ of $\L^2([-\kappa,\kappa])$ and
$\{\phi_\bfj:\bfj\in \Zset^d\}$ of $\L^2([-\kappa,\kappa]^d)$:
\begin{align*}
\varphi_j(x) := \frac{1}{\sqrt{\kappa}}\begin{cases}
\sin\paren{\pi j x/\kappa},&j>0,\\
1/\sqrt{2},& j=0,\\
\cos\paren{\pi j x/\kappa},&j<0,
\end{cases}
\qquad 
\phi_{\bfj}(x_1,\cdots,x_d):= \prod_{l=1}^d \varphi_{j_l}(x_l)\,.
\end{align*}
We introduce the norm $\|\mathfrak{g}\|_{H^s_{\rm per}} = \big(\sum_{j\in\Zset^d}(1+|\bfj|^2)^s
|\langle \mathfrak{g},\phi_{\bfj}\rangle|^2\big)^{1/2}$ and 
the periodic Sobolev space $H^s_{\rm per} \big(\left[-\kappa,\kappa\right]^d\big):=\{\mathfrak{g}\in \L^2([-\kappa,\kappa]^d)~\big|~\|\mathfrak{g}\|_{H^s_{\rm per}}<\infty\}$.  
The embedding $J:H^s_0\left(\left[-\kappa, \kappa\right]^d\right)
\hookrightarrow H^s_{\rm per} \left(\left[-\kappa,\kappa\right]^d\right)$ 
is well defined and continuous as the norms of both spaces are equivalent
(see e.g.\ \cite[Exercise 1.13]{w87}), so the extension operator
\[
E_{\rm ext}:= J\circ E : H^s \left(\manifold\right) \longrightarrow H^s_{\rm per} \left(\left[-\kappa, \kappa\right]^d\right).
\]
is continuous. In particular, 
\[
E_{\rm ext} \left(B_s \left(R\right)\right) \subset\left\{\g \in H^s_{\rm per} \left(\left[-\kappa, \kappa\right]^d\right) ~\big|~ \left\Vert \g \right\Vert_{H^s_{\rm per} \left(\left[-\kappa, \kappa\right]^d\right)} 
\leq \Cext R\right\}\quad \mbox{with }
\Cext:= \left\Vert E_{\rm ext}\right\Vert
\]
and \eqref{eq:sum_representation} holds true with 
$\beta_\bfj := \left<\g,\phi_\bfj\right>$ and 
$\gamma_\bfj := \big(1 + \left|\bfj\right|^2 \big)^{-s/2}$. 
Moreover, 
as $\|\phi_{\bfj}^2\|_{\infty}\leq \kappa^{-d}$ for all $\bfj\in\Zset^d$
we obtain
\[
\sum\limits_{\bfj\in \mathbb Z^d} \gamma_\bfj^2 \int\limits_{\manifold} \phi_\bfj^2 \gdag \,\mathrm d x 
\leq \Cgdag \int\limits_{\manifold} \gdag \,\mathrm d x\quad
\mbox{with }
\Cgdag := \kappa^{-d}
\sum\limits_{\bfj\in \mathbb Z^d} \left(1 + \left|\bfj\right|^2 \right)^{-s}\,,
\]
and majorization of the sum by an integral shows that 
$\Cgdag<\infty$ as $s > d/2$. 
Therefore,  \eqref{eq:cond_gamma} holds true, and the proof is complete.
\end{proof}

\begin{lem}\label{lem:ew}
Under the assumptions of Lemma \ref{lem:sum} we have 
\[
\EW{\tilde Z} \leq \frac{\Cgdag\Cext R}{\sqrt{t}}\|\gdag\|_{\L^1(\manifold)}\,.
\]
\end{lem}
\begin{proof}
With the help of Lemma \ref{lem:sum} we can now insert \eqref{eq:sum_representation} and apply H\"older's inequality for sums to find
\begin{align*}
\tilde Z &\leq \sup\limits_{\sum\limits_{\bfj\in J} \left(\frac{\beta_\bfj}{\gamma_\bfj}\right)^2 \leq (\Cext R)^2} \left| \sum\limits_{\bfj\in J} \frac{\beta_\bfj}{\gamma_\bfj} \int\limits_{\manifold} \gamma_\bfj \phi_\bfj \left(\mathrm d G_t - \gdag\,\mathrm d x\right)\right| \\[0.1cm]
&\leq \Cext R \sqrt{\sum\limits_{\bfj\in J} \gamma_\bfj^2 \bigg(\int\limits_{\manifold} \phi_\bfj \left(\mathrm d G_t - \gdag\,\mathrm d x\right)\bigg)^2}
\end{align*}
where we used that the functions $\phi_\bfj$ are real-valued. Hence by Jensen's inequality
\begin{align}
\EW{\tilde Z} & \leq \sqrt{\EW {\tilde{Z}^2}}\leq \Cext R \sqrt{\sum\limits_{\bfj\in \mathcal{J}} \gamma_\bfj^2 \EW {\bigg(\int\limits_{\manifold} \phi_\bfj \left(\mathrm d G_t - \gdag\,\mathrm d x\right)\bigg)^2}} \label{eq:EW_est_1}.
\end{align}
Using \eqref{eq:integralPoisson} we obtain
\[
\EW{\bigg(\int\limits_{\manifold} \phi_\bfj \left(\mathrm d G_t - \gdag\,\mathrm d x\right)\bigg)^2} 
= \frac{1}{t^2} \EW {\bigg(\int\limits_{\manifold} \phi_\bfj 
\left(t \,\mathrm d G_t - t\gdag\,\mathrm d x\right)\bigg)^2} 
= \frac1t \int\limits_{\manifold} \phi_\bfj^2 \gdag \,\mathrm d x\,
\]
and plugging this into \eqref{eq:EW_est_1} and using \eqref{eq:cond_gamma}
we obtain
\[
\EW{\tilde Z} \leq \frac{\Cext R}{\sqrt{t}} \sqrt{\sum\limits_{\bfj\in \mathcal{J}} 
\gamma_\bfj^2 \int\limits_{\manifold} \phi_\bfj^2 \gdag \,\mathrm d x}
\leq \frac{\sqrt{\Cgdag}\Cext R}{\sqrt{t}}\sqrt{\|\gdag\|_{\L^1(\manifold)}}\,.
\]
\end{proof}

\begin{proof}[Proof of Theorem \ref{thm:pie}]
By Sobolev's embedding theorem the embedding operator $E_\infty :H^s \left(\manifold\right) \hookrightarrow \L^\infty \left(\manifold\right)$
is well defined and continuous, so 
\begin{equation}\label{eq:bound_frakg}
\left\Vert \g \right\Vert_{\L^\infty\left(\manifold\right)} 
\leq R\left\Vert E_{\infty} \right\Vert   \qquad\text{for all}\qquad \g \in B_s\left(R\right)\,.
\end{equation}
Now we choose a countable subset $\left\{\g_a\right\}_{a \in A} \subset B_s\left(R\right)$ which is dense in $B_s\left(R\right)$ w.r.t.\ the $H^s$-norm,
and hence also the $\L^\infty$-norm and set $N = t G_t$ and $\,\mathrm d \nu = t \gdag \,\mathrm d x$ in  Lemma \ref{lem:concie} to obtain 
\begin{equation}\label{eq:concie}
\Prob {\tilde{Z} \geq \left(1+ \varepsilon\right) \EW{\tilde{Z}} + 
\frac{\sqrt{12 v_0 \bar \rho}}{t} + 
\frac{\kappa \left(\varepsilon\right) \left\Vert E_{\infty} \right\Vert  R \bar \rho}{t}} \leq \exp\left(-\bar \rho\right)
\end{equation}
for all $\bar \rho > 0$.  Choosing $\varepsilon=1$ and using Lemma \ref{lem:ew} 
and the simple estimate 
\[
\tilde{v}_0  \leq t R^2 \left\Vert E_{\infty} \right\Vert^2 \left\Vert \gdag\right\Vert_{\L^1 \left(\manifold\right)}\,,
\] 
yields 
\begin{equation}\label{eq:concie_proof_end}
\Prob{\tilde Z\leq \frac{C_1 R}{\sqrt{t}} 
+ \frac{ C_2 R \sqrt{\bar \rho}}{\sqrt{t}} 
+ \frac{C_3 R \bar \rho }{t}}\geq 1 -\exp\left(-\bar \rho\right)
\end{equation}
for all $\bar \rho, t > 0$ with $C_1:= 2\sqrt{\Cgdag} \Cext \sqrt{\|\gdag\|_{\L^1(\manifold)}}$, 
 $C_2 := \sqrt{12}\left\Vert E_{\infty} \right\Vert \sqrt{\left\Vert \gdag\right\Vert_{\L^1 \left(\manifold\right)}}$ and $C_3 :=\left(32 + \frac54\right)\left\Vert E_{\infty} \right\Vert$. If $t, \bar \rho\geq 1$, we have $\frac{1}{t} \leq \frac{1}{\sqrt{t}}$ 
and $\sqrt{\rho}\leq \rho$, so
\[
\Prob{\tilde Z \leq \left(C_1 + C_2 + C_3\right) \frac{\bar \rho R}{\sqrt{t}} }\geq 1 -\exp\left(-\bar\rho\right)\qquad\mbox{for }\bar \rho, t \geq 1\,.
\]
Setting $\Cconc := \max\left\{C_1 + C_2 + C_3,1\right\}$ and $\rho := \bar \rho R\Cconc$ this shows the assertion. 
\end{proof}

\section*{Acknowledgments}
We would like to thank Patricia Reynaud-Bouret for fruitful discussions on the concentration inequality. Financial support by the German Research Foundation DFG through the Research Training Group 1023 and CRC 755 is gratefully acknowledged. 

\small
\bibliography{werner_hohage_12}{}
\bibliographystyle{plain}
\end{document}